\documentclass[11pt,reqno]{amsart}
\calclayout
\usepackage{amssymb,mathtools,enumerate}
\usepackage[hidelinks]{hyperref}
\newtheorem{Thm}{Theorem}[section]
\newtheorem{Lem}[Thm]{Lemma}
\newtheorem{Cor}[Thm]{Corollary}

\newcommand{\ZZ}{\mathbb{Z}}

\begin{document}
\title[A double-sum analogue of Whipple's summation formula]{A double-sum analogue of Whipple's summation formula and a restricted sum formula for multiple binomial sums}
\author{Takumi Maesaka}
\address{Faculty of Mathematics, Kyushu University, Motooka 744, Nishi-ku, Fukuoka, 819-0395, Japan}
\email{nozaki.takumi.912@s.kyushu-u.ac.jp}
\begin{abstract}
We prove a double-sum analogue of Whipple's ${}_3F_2$ summation formula. As an application, we establish a restricted sum formula for multiple binomial sums, which arise as special cases of Igarashi's parametrized multiple zeta series.
\end{abstract}
\maketitle

\section{Introduction}
For complex numbers $a,a_1,\dots,a_r$ and a nonnegative integer $n$, define the shifted factorials by
\begin{align*}
(a)_n&\coloneqq \begin{cases}
1&n=0\\
a(a+1)\cdots(a+n-1)&n>0
\end{cases}\\
(a_1,\dots,a_r)_n&\coloneqq (a_1)_n\cdots(a_r)_n.
\end{align*}
Whipple's summation formula is the following classical identity for the generalized hypergeometric series ${}_3F_2$.
\begin{Thm}[Whipple \cite{w}]\label{thm:whipple}
If $a+b=1$, $d+e=2c+1$, and $\Re(c)> 0$, then the identity
\begin{align*}
\sum_{0\leq n}\frac{(a,b,c)_n}{n!(d,e)_n}&=\frac{2^{1-2c}\pi\Gamma(d)\Gamma(e)}{\Gamma\left(\frac{a+d}2\right)\Gamma\left(\frac{a+e}2\right)\Gamma\left(\frac{b+d}2\right)\Gamma\left(\frac{b+e}2\right)}
\end{align*}
holds.
\end{Thm}
We prove the following double-sum analogue of this formula.
\begin{Thm}\label{thm:double_whipple}
Let $a,b,c,d,e$ be complex numbers with $\Re(c) < 1$, and assume that all the expressions below are well-defined.
\begin{enumerate}[(i)]
\item If $a+b=1$ and $d+e=2c+1$, then the identity
\begin{align*}
\sum_{0\leq n\leq m}\frac{(a,b)_n(c)_{n+1}}{n!(d,e)_{n+1}}\frac{m!(d,e)_m}{(a,b,c)_{m+1}}&=\frac{2c}{\sin\pi c}\sum_{0\leq n}\frac{(-1)^n(2n+1)\cos\frac{\pi(d-n)}2\cos\frac{\pi(e-n)}2}{(n+a)(n+b)(n+d)(n+1-d)(n+e)(n+1-e)}
\end{align*}
holds.
\item If $a+b=1$ and $d+e=2c$, then the identity
\begin{align*}
&\sum_{0\leq n\leq m}\frac{(n+m+2c)(a,b,c)_n}{n!(d,e)_{n+1}}\frac{m!(d,e)_m}{(a,b,c)_{m+1}}\\
&=\frac 1{d-e}\sum_{0\leq n}\frac{2n+1}{(n+a)(n+b)}\left(\frac 1{(n+e)(n+1-e)}-\frac 1{(n+d)(n+1-d)}\right)\\
&\qquad+\frac{\sin\frac{\pi(d-e)}2}{(d-e)\sin\pi c}\sum_{0\leq n}\frac{(-1)^n(2n+1)}{(n+a)(n+b)}\left(\frac 1{(n+e)(n+1-e)}+\frac 1{(n+d)(n+1-d)}\right)
\end{align*}
holds.
\end{enumerate}
\end{Thm}

Igarashi \cite{i} introduced the following parametrized multiple zeta series.
\begin{align*}
Z(k_1,\dots,k_r;\alpha):=\sum_{0\leq n_1<\cdots<n_r}\frac{(\alpha)_{n_1}}{n_1!}\frac{1}{(n_1+\alpha)^{k_1}\cdots (n_r+\alpha)^{k_r}}\frac{n_r!}{(\alpha)_{n_r}},
\end{align*}
where $(k_1,\dots,k_r)\in (\ZZ_{>0})^r, k_r\geq 2$ and $\Re(\alpha)>0$. Set $Z(k_1,\dots,k_r)\coloneqq Z\left(k_1,\dots,k_r;\frac 12\right)$. As an application of Theorem \ref{thm:double_whipple}, we establish the following restricted sum formula.
\begin{Thm}\label{thm:main}
For positive integers $k\geq r\geq 2$, the identity
\begin{align*}
&\sum_{\substack{1\leq k_1,\dots,k_r\\k_1+\cdots+k_r=k}}(Z(2k_1,\dots,2k_r)+Z(2k_1-1,2k_2,\dots,2k_{r-1},2k_r+1))\\
&=\frac 1{2^{2r-3}}\sum_{j=0}^{\lfloor{\frac{r-1}2}\rfloor}\frac{(-4\pi^2)^j}{(2j)!}\binom{2r-2j-2}{r-1}Z(2k-2j).
\end{align*}
holds.
\end{Thm}
This result is analogous to the following restricted sum formulas for multiple zeta values and multiple $t$-values, established by Hoffman \cite{h} and Zhao \cite{z}, respectively:
\begin{align*}
\sum_{\substack{1\leq k_1,\dots,k_r\\k_1+\cdots+k_r=k}}\zeta(2k_1,\dots,2k_r)&=\frac 1{2^{2r-2}}\sum_{j=0}^{\lfloor\frac{r-1}2\rfloor}\frac{(-4\pi^2)^j}{(2j+1)!}\binom{2r-2j-1}r\zeta(2k-2j),\\
\sum_{\substack{1\leq k_1,\dots,k_r\\k_1+\cdots+k_r=k}}t(2k_1,\dots,2k_r)&=\frac 1{2^{2r-2}r}\sum_{j=0}^{\lfloor\frac{r-1}2\rfloor}\frac{(-\pi^2)^j}{(2j)!}\binom{2r-2j-2}{r-1}t(2k-2j),
\end{align*}
where
\begin{align*}
\zeta(k_1,\dots,k_r)&:=\sum_{0<n_1<\cdots<n_r}\frac 1{n_1^{k_1}\cdots n_r^{k_r}}\\
t(k_1,\dots,k_r)&:=\sum_{0<n_1<\cdots<n_r}\frac 1{(2n_1-1)^{k_1}\cdots (2n_r-1)^{k_r}}.
\end{align*}
Combining Zhao's formula with the identity $Z(k)=2^kt(k)$ for $k\geq 2$, we see that Theorem 1.3 is equivalent to the following identity.
\begin{Thm}
For positive integers $k\geq r\geq 2$, the identity
\begin{align*}
&\sum_{\substack{1\leq k_1,\dots,k_r\\k_1+\cdots+k_r=k}}(Z(2k_1,\dots,2k_r)+Z(2k_1-1,2k_2,\dots,2k_{r-1},2k_r+1))\\
&=2^{2k+1}r\sum_{\substack{1\leq k_1,\dots,k_r\\k_1+\cdots+k_r=k}}t(2k_1,\dots,2k_r)
\end{align*}
holds.
\end{Thm}

Setting $k=r$ and using the following identity (see Hoffman \cite{h2})
\[
t(\{2\}^r)=\frac{\pi^{2r}}{2^{2r}(2r)!},
\]
we obtain the following corollary.
\begin{Cor}
Let $r$ be a positive integer greater than or equal to $2$. Then the following identity holds:
\begin{align*}
&Z(\{2\}^r)+Z(1,\{2\}^{r-2},3)=\frac{\pi^{2r}}{(2r-1)!}.
\end{align*}
Here, $\{2\}^n$ denotes a string of $n$ twos.
\end{Cor}

This paper is organized as follows. In Section 2, we prove Theorem \ref{thm:double_whipple} using generalizations of Whipple's and Watson's summation formulas due to Lavoie, Grondin, and Rathie \cite{lgr,lgr2}. In Section 3, we apply Theorem \ref{thm:double_whipple} to prove Theorem \ref{thm:main}.

\section{Proof of Theorem \ref{thm:double_whipple}}

Lavoie, Grondin, and Rathie \cite{lgr, lgr2} obtained generalizations of Whipple's and Watson's summation formulas.
\begin{Lem}[Watson \cite{wa}; Lavoie--Grondin--Rathie \cite{lgr,lgr2}]\label{lem:lgr}
\begin{enumerate}[(i)]
\item If $a+b=1$, $d+e=2c$, and $\Re(c)>1$, then the identity
\begin{align*}
\sum_{0\leq n}\frac{(a,b,c)_n}{n!(d,e)_n}&=\frac{2^{2-2c}\pi\Gamma(d)\Gamma(e)}{c-1}\left(\frac 1{\Gamma\left(\frac{d-a}2\right)\Gamma\left(\frac{d-b}2\right)\Gamma\left(\frac{e+a}2\right)\Gamma\left(\frac{e+b}2\right)}+\frac 1{\Gamma\left(\frac{e-a}2\right)\Gamma\left(\frac{e-b}2\right)\Gamma\left(\frac{d+a}2\right)\Gamma\left(\frac{d+b}2\right)}\right)
\end{align*}
holds.
\item If $a+b=1$, $d+e=2c+2$, and $\Re(c)>-1$, then the identity
\begin{align*}
\sum_{0\leq n}\frac{(a,b,c)_n}{n!(d,e)_n}&=\frac{2^{1-2c}\pi\Gamma(d)\Gamma(e)}{d-e}\left(\frac 1{\Gamma\left(\frac{d-a}2\right)\Gamma\left(\frac{d-b}2\right)\Gamma\left(\frac{e+a}2\right)\Gamma\left(\frac{e+b}2\right)}-\frac 1{\Gamma\left(\frac{e-a}2\right)\Gamma\left(\frac{e-b}2\right)\Gamma\left(\frac{d+a}2\right)\Gamma\left(\frac{d+b}2\right)}\right)
\end{align*}
holds.
\item If $\Re(2c-a-b)>-1$, then the identity
\begin{align*}
\sum_{0\leq n}\frac{(a,b,c)_n}{n!\left(\frac{a+b+1}2,2c\right)_n}&=\frac{\Gamma\left(\frac 12\right)\Gamma\left(c+\frac 12\right)\Gamma\left(\frac{a+b+1}2\right)\Gamma\left(c-\frac{a+b-1}2\right)}{\Gamma\left(\frac{a+1}2\right)\Gamma\left(c-\frac{a-1}2\right)\Gamma\left(\frac{b+1}2\right)\Gamma\left(c-\frac{b-1}2\right)}
\end{align*}
holds.
\item If $\Re(2c-a-b)>0$, then the identity
\begin{align*}
\sum_{0\leq n}\frac{(a,b,c)_n}{n!\left(\frac{a+b}2,2c\right)_n}&=\frac{\Gamma\left(\frac 12\right)\Gamma\left(c+\frac 12\right)\Gamma\left(\frac{a+b}2\right)\Gamma\left(c-\frac{a+b}2\right)}{\Gamma\left(\frac a2\right)\Gamma\left(c-\frac a2\right)\Gamma\left(\frac{b+1}2\right)\Gamma\left(c-\frac{b-1}2\right)}+\frac{\Gamma\left(\frac 12\right)\Gamma\left(c+\frac 12\right)\Gamma\left(\frac{a+b}2\right)\Gamma\left(c-\frac{a+b}2\right)}{\Gamma\left(\frac{a+1}2\right)\Gamma\left(c-\frac{a-1}2\right)\Gamma\left(\frac b2\right)\Gamma\left(c-\frac b2\right)}
\end{align*}
holds.
\item If $\Re(2c-a-b)>-2$, then the identity
\begin{align*}
\sum_{0\leq n}\frac{(a,b,c)_n}{n!\left(\frac{a+b+2}2,2c\right)_n}&=\frac{a+b}{a-b}\left(\frac{\Gamma\left(\frac 12\right)\Gamma\left(c+\frac 12\right)\Gamma\left(\frac{a+b}2\right)\Gamma\left(c-\frac{a+b}2\right)}{\Gamma\left(\frac a2\right)\Gamma\left(c-\frac a2\right)\Gamma\left(\frac{b+1}2\right)\Gamma\left(c-\frac{b-1}2\right)}-\frac{\Gamma\left(\frac 12\right)\Gamma\left(c+\frac 12\right)\Gamma\left(\frac{a+b}2\right)\Gamma\left(c-\frac{a+b}2\right)}{\Gamma\left(\frac{a+1}2\right)\Gamma\left(c-\frac{a-1}2\right)\Gamma\left(\frac b2\right)\Gamma\left(c-\frac b2\right)}\right)
\end{align*}
holds.
\end{enumerate}
\end{Lem}

\begin{Lem}\label{lem:main}
Let $N$ be a nonnegative integer, and let $c,d,e$ be complex numbers satisfying $d+e=2c$ and $\Re(c)<1$. Assume that all
the expressions below are well-defined. Then the following identities hold:
\begin{align*}
\sum_{0\leq n\leq N}\frac{(c)_n}{n!(d,e)_{n+1}}\frac{(-1)^n(n+N)!}{(N-n)!}&=\frac 1{2(d-e)}(A_N-B_N),\\
\sum_{0\leq n\leq N}\frac{(c)_{n+1}}{n!(d,e)_{n+1}}\frac{(-1)^n(n+N)!}{(N-n)!}&=\frac{(-1)^N}{4}(A_N+B_N),\\
\sum_{N\leq m}\frac 1{(m+N+1)!(m-N)!}\frac{m!(d,e)_m}{(c)_{m+1}}&=\frac{(-1)^N}{(d-e)\sin\pi c}(C_N-D_N),\\
\sum_{N\leq m}\frac 1{(m+N+1)!(m-N)!}\frac{m!(d,e)_m}{(c)_{m}}&=\frac 1{2\sin\pi c}(C_N+D_N),
\end{align*}

where
\begin{align*}
A_N&\coloneqq\begin{cases}
\displaystyle\frac{\left(1-\frac d2,\frac{1-e}2\right)_k}{\left(\frac{d+1}2\right)_k\left(\frac e2\right)_{k+1}}& N=2k\\
\displaystyle\frac{\left(1-\frac d2\right)_k\left(\frac{1-e}2\right)_{k+1}}{\left(\frac e2,\frac{d+1}2\right)_{k+1}}&N=2k+1
\end{cases},\\
B_N&\coloneqq\begin{cases}
\displaystyle\frac{\left(1-\frac e2,\frac{1-d}2\right)_k}{\left(\frac{e+1}2\right)_k\left(\frac d2\right)_{k+1}} & N=2k\\
\displaystyle\frac{\left(1-\frac e2\right)_k\left(\frac{1-d}2\right)_{k+1}}{\left(\frac d2,\frac{e+1}2\right)_{k+1}} & N=2k+1
\end{cases},\\
C_N&\coloneqq\begin{cases}
\displaystyle\frac{\sin\frac{\pi d}2\cos\frac{\pi e}2\left(\frac{d+1}2,\frac{e}2\right)_k}{\left(1-\frac d2\right)_k\left(\frac{1-e}2\right)_{k+1}} & N=2k\\
\displaystyle\frac{\sin\frac{\pi d}2\cos\frac{\pi e}2\left(\frac e2\right)_{k+1}\left(\frac{d+1}2\right)_k}{\left(\frac{1-e}2,1-\frac d2\right)_{k+1}}&N=2k+1
\end{cases},\\
D_N&\coloneqq\begin{cases}
\displaystyle\frac{\sin\frac{\pi e}2\cos\frac{\pi d}2\left(\frac{e+1}2,\frac d2\right)_k}{\left(1-\frac e2\right)_k\left(\frac{1-d}2\right)_{k+1}}&N=2k\\
\displaystyle\frac{\sin\frac{\pi e}2\cos\frac{\pi d}2\left(\frac d2\right)_{k+1}\left(\frac{e+1}2\right)_k}{\left(\frac{1-d}2,1-\frac e2\right)_{k+1}} & N=2k+1
\end{cases}.
\end{align*}
\end{Lem}

\begin{proof}
We prove the first identity in the case $N=2k$. Applying Lemma \ref{lem:lgr} (ii), we obtain
\begin{align*}
\sum_{0\leq n\leq N}\frac{(c)_n}{n!(d,e)_{n+1}}\frac{(-1)^n(n+N)!}{(N-n)!}&=\frac 1{de}\sum_{0\leq n\leq N}\frac{(-N,N+1,c)_n}{n!(d+1,e+1)_{n}}\\
&=\frac 1{de}\frac{2^{1-2c}\Gamma(d+1)\Gamma(e+1)}{d-e}\\
&\qquad\cdot\bigg(\frac 1{\Gamma\left(\frac{d+1}2+k\right)\Gamma\left(\frac{d}2-k\right)\Gamma\left(\frac{e+1}2-k\right)\Gamma\left(\frac{e}2+k+1\right)}\\
&\qquad\qquad-\frac 1{\Gamma\left(\frac{e+1}2+k\right)\Gamma\left(\frac{e}2-k\right)\Gamma\left(\frac{d+1}2-k\right)\Gamma\left(\frac{d}2+k+1\right)}\bigg)\\
&=\frac{2^{1-2c}\Gamma(d)\Gamma(e)}{(d-e)\Gamma\left(\frac d2\right)\Gamma\left(\frac{d+1}2\right)\Gamma\left(\frac e2\right)\Gamma\left(\frac{e+1}2\right)}\\
&\qquad\cdot\bigg(\frac{\left(\frac d2-k,\frac{e+1}2-k\right)_k}{\left(\frac{d+1}2\right)_k\left(\frac e2\right)_{k+1}}-\frac {\left(\frac e2-k,\frac{d+1}2-k\right)_k}{\left(\frac{e+1}2\right)_k\left(\frac d2\right)_{k+1}}\bigg).
\end{align*}
Using the identities $\Gamma\left(\frac x2\right)\Gamma\left(\frac{x+1}2\right)=2^{1-x}\sqrt{\pi}\Gamma(x)$ and $(x-k)_k=(-1)^k(1-x)_k,$ we obtain the first identity. All the remaining cases are proved similarly, so we omit the details.
\end{proof}

\begin{Lem}\label{lem:partial_fraction}
If $a+b=1$, the identity
\[
\frac{(a,b)_n}{(a,b)_{m+1}}=\sum_{k=n}^m\frac{(-1)^k(2k+1)}{(k+a)(k+b)}\frac{(-1)^n(n+k)!}{(k-n)!(m-k)!(m+k+1)!}
\]
holds.
\end{Lem}

\begin{proof}
Let $a=\frac 12+s$. A partial fraction decomposition with respect to $s^2$ gives
\begin{align*}
\frac{(a,b)_n}{(a,b)_{m+1}}&=\frac{\left(\frac 12-s,\frac 12+s\right)_n}{\left(\frac 12-s,\frac 12+s\right)_{m+1}}\\
&=\prod_{n\leq k\leq m}\frac 1{\left(k+\frac 12\right)^2-s^2}\\
&=\sum_{k=n}^m\frac{(-1)^k(2k+1)}{\left(k+\frac 12\right)^2-s^2}\frac{(-1)^n(n+k)!}{(k-n)!(m-k)!(m+k+1)!}\\
&=\sum_{k=n}^m\frac{(-1)^k(2k+1)}{(k+a)(k+b)}\frac{(-1)^n(n+k)!}{(k-n)!(m-k)!(m+k+1)!}.
\end{align*}
\end{proof}

\begin{proof}[Proof of Theorem \ref{thm:double_whipple}]
We first prove part (i). Applying Lemma \ref{lem:partial_fraction}, we obtain
\begin{align*}
&\sum_{0\leq n\leq m}\frac{(a,b)_n(c)_{n+1}}{n!(d,e)_{n+1}}\frac{m!(d,e)_m}{(a,b,c)_{m+1}}\\
&=\sum_{0\leq n\leq m}\frac{(c)_{n+1}}{n!(d,e)_{n+1}}\frac{m!(d,e)_m}{(c)_{m+1}}\sum_{k=n}^m\frac{(-1)^k(2k+1)}{(k+a)(k+b)}\frac{(-1)^n(n+k)!}{(k-n)!(m-k)!(m+k+1)!}\\
&=\sum_{0\leq k}\frac{(-1)^k(2k+1)}{(k+a)(k+b)}\sum_{0\leq n\leq k}\frac{(c)_{n+1}}{n!(d,e)_{n+1}}\frac{(-1)^n(n+k)!}{(k-n)!}\sum_{k\leq m}\frac{m!(d,e)_m}{(c)_{m+1}}\frac{1}{(m-k)!(m+k+1)!}
\end{align*}

By Lemma \ref{lem:lgr} (iii) and Theorem \ref{thm:whipple}, we obtain
\begin{align*}
\sum_{k\leq m}\frac 1{(m+k+1)!(m-k)!}\frac{m!(d,e)_m}{(c)_{m+1}}&=\frac{\pi^2}{\sin\pi c}\frac{(d,e)_k}{2^{2k+1}\Gamma\left(\frac{d+k+1}{2}\right)\Gamma\left(\frac{e+k+1}2\right)\Gamma\left(\frac{3-d+k}2\right)\Gamma\left(\frac{3-e+k}2\right)},\\
\sum_{0\leq n}\frac{(c)_{n+1}}{n!(d,e)_{n+1}}\frac{(-1)^n(n+k)!}{(k-n)!}&=\frac{2^{-1-2c}\pi c\Gamma(d)\Gamma(e)}{\Gamma\left(\frac{1+d-k}2\right)\Gamma\left(\frac{1+e-k}2\right)\Gamma\left(\frac{d+k+2}2\right)\Gamma\left(\frac{e+k+2}2\right)}.
\end{align*}
Hence,
\begin{align*}
&\sum_{0\leq n\leq m}\frac{(a,b)_n(c)_{n+1}}{n!(d,e)_{n+1}}\frac{m!(d,e)_m}{(a,b,c)_{m+1}}\\
&=\sum_{0\leq k}\frac{(-1)^k(2k+1)}{(k+a)(k+b)}\frac{\pi^2}{\sin\pi c}\frac{(d,e)_k}{2^{2k+1}\Gamma\left(\frac{d+k+1}{2}\right)\Gamma\left(\frac{e+k+1}2\right)\Gamma\left(\frac{3-d+k}2\right)\Gamma\left(\frac{3-e+k}2\right)}\\
&\qquad\cdot\frac{2^{-1-2c}\pi c\Gamma(d)\Gamma(e)}{\Gamma\left(\frac{1+d-k}2\right)\Gamma\left(\frac{1+e-k}2\right)\Gamma\left(\frac{d+k+2}2\right)\Gamma\left(\frac{e+k+2}2\right)}\\
&=\frac{2c}{\sin\pi c}\sum_{0\leq k}\frac{(-1)^k(2k+1)\cos\frac{\pi(d-k)}2\cos\frac{\pi(e-k)}2}{(k+a)(k+b)(k+d)(k+1-d)(k+e)(k+1-e)}.
\end{align*}

We next prove part (ii). A similar argument gives
\begin{align*}
&\sum_{0\leq n\leq m}\frac{(n+m+2c)(a,b,c)_n}{n!(d,e)_{n+1}}\frac{m!(d,e)_m}{(a,b,c)_{m+1}}\\
&=\sum_{0\leq n\leq m}\frac{(a,b)_n(c)_{n+1}}{n!(d,e)_{n+1}}\frac{m!(d,e)_m}{(a,b,c)_{m+1}}+\sum_{0\leq n\leq m}\frac{(a,b,c)_n}{n!(d,e)_{n+1}}\frac{m!(d,e)_m}{(a,b)_{m+1}(c)_m}\\
&=\sum_{0\leq k}\frac{(-1)^k(2k+1)}{(k+a)(k+b)}\sum_{0\leq n\leq k}\frac{(c)_{n+1}}{n!(d,e)_{n+1}}\frac{(-1)^n(n+k)!}{(k-n)!}\sum_{k\leq m}\frac{m!(d,e)_m}{(c)_{m+1}}\frac{1}{(m-k)!(m+k+1)!}\\
&\qquad+\sum_{0\leq k}\frac{(-1)^k(2k+1)}{(k+a)(k+b)}\sum_{0\leq n\leq k}\frac{(c)_{n}}{n!(d,e)_{n+1}}\frac{(-1)^n(n+k)!}{(k-n)!}\sum_{k\leq m}\frac{m!(d,e)_m}{(c)_{m}}\frac{1}{(m-k)!(m+k+1)!}
\end{align*}
Applying Lemma \ref{lem:main} gives
\begin{align*}
&\sum_{0\leq n\leq m}\frac{(n+m+2c)(a,b,c)_n}{n!(d,e)_{n+1}}\frac{m!(d,e)_m}{(a,b,c)_{m+1}}\\
&=\frac 1{4(d-e)\sin\pi c}\sum_{0\leq k}\frac{(-1)^k(2k+1)}{(k+a)(k+b)}(A_k+B_k)(C_k-D_k)\\
&\qquad+\frac 1{4(d-e)\sin\pi c}\sum_{0\leq k}\frac{(-1)^k(2k+1)}{(k+a)(k+b)}(A_k-B_k)(C_k+D_k)\\
&=\frac 1{2(d-e)\sin\pi c}\sum_{0\leq k}\frac{(-1)^k(2k+1)}{(k+a)(k+b)}(A_kC_k-B_kD_k).
\end{align*}
Moreover,
\begin{align*}
&\sum_{0\leq k}\frac{(-1)^k(2k+1)}{(k+a)(k+b)}(A_kC_k-B_kD_k)\\
&=\sum_{0\leq k}\frac{4k+1}{(2k+a)(2k+b)}\left(\frac{\sin\frac{\pi d}2\cos\frac{\pi e}2}{\left(k+\frac e2\right)\left(k+\frac{1-e}2\right)}-\frac{\sin\frac{\pi e}2\cos\frac{\pi d}2}{\left(k+\frac{d}2\right)\left(k+\frac{1-d}2\right)}\right)\\
&\qquad-\sum_{0\leq k}\frac{4k+3}{(2k+1+a)(2k+1+b)}\left(\frac{\sin\frac{\pi d}2\cos\frac{\pi e}2}{\left(k+\frac{d+1}2\right)\left(k+1-\frac d2\right)}-\frac{\sin\frac{\pi e}2\cos\frac{\pi d}2}{\left(k+\frac{e+1}2\right)\left(k+1-\frac e2\right)}\right)\\
&=\frac{\sin\pi c}2\sum_{0\leq k}\frac{4k+1}{(2k+a)(2k+b)}\left(\frac{1}{\left(k+\frac e2\right)\left(k+\frac{1-e}2\right)}-\frac{1}{\left(k+\frac{d}2\right)\left(k+\frac{1-d}2\right)}\right)\\
&\qquad-\frac{\sin\pi c}2\sum_{0\leq k}\frac{4k+3}{(2k+1+a)(2k+1+b)}\left(\frac{1}{\left(k+\frac{d+1}2\right)\left(k+1-\frac d2\right)}-\frac{1}{\left(k+\frac{e+1}2\right)\left(k+1-\frac e2\right)}\right)\\
&\qquad+\frac{\sin\frac{\pi(d-e)}2}{2}\sum_{0\leq k}\frac{4k+1}{(2k+a)(2k+b)}\left(\frac{1}{\left(k+\frac e2\right)\left(k+\frac{1-e}2\right)}+\frac{1}{\left(k+\frac{d}2\right)\left(k+\frac{1-d}2\right)}\right)\\
&\qquad-\frac{\sin\frac{\pi(d-e)}2}{2}\sum_{0\leq k}\frac{4k+3}{(2k+1+a)(2k+1+b)}\left(\frac{1}{\left(k+\frac{d+1}2\right)\left(k+1-\frac d2\right)}+\frac{1}{\left(k+\frac{e+1}2\right)\left(k+1-\frac e2\right)}\right)\\
&=2\sin\pi c\sum_{0\leq k}\frac{2k+1}{(k+a)(k+b)}\left(\frac{1}{\left(k+e\right)\left(k+1-e\right)}-\frac{1}{\left(k+d\right)\left(k+1-d\right)}\right)\\
&\qquad+2\sin\frac{\pi(d-e)}2\sum_{0\leq k}\frac{(-1)^k(2k+1)}{(k+a)(k+b)}\left(\frac{1}{\left(k+e\right)\left(k+1-e\right)}+\frac{1}{\left(k+d\right)\left(k+1-d\right)}\right).
\end{align*}
This completes the proof.
\end{proof}

\section{Proof of Theorem \ref{thm:main}}
\begin{Lem}\label{lem:1}
The following identity holds as an identity of formal power series in $s$ and $t$:
\begin{align*}
&\sum_{2\leq r\leq k}s^{2k-4}t^{2r-4}\sum_{\substack{1\leq k_1,\dots,k_r\\k_1+\cdots+k_r=k}}(Z(2k_1,\dots,2k_r)+Z(2k_1-1,2k_2,\dots,2k_{r-1},2k_r+1))\\
&=\frac{\pi\sin\left(\pi s\sqrt{1-t^2}\right)}{s^3t^2\sqrt{1-t^2}\cos\pi s}-\frac{\pi\tan\pi s}{s^3t^2}.
\end{align*}
\end{Lem}

\begin{proof}
By direct calculation, we obtain
\begin{align*}
&\sum_{2\leq r\leq k}s^{2k-4}t^{2r-4}\sum_{\substack{1\leq k_1,\dots,k_r\\k_1+\cdots+k_r=k}}(Z(2k_1,\dots,2k_r)+Z(2k_1-1,2k_2,\dots,2k_{r-1},2k_r+1))\\
&=\sum_{2\leq r\leq k}s^{2k-2r}(st)^{2r-4}\sum_{\substack{1\leq k_1,\dots,k_r\\k_1+\cdots+k_r=k}}\sum_{0\leq n_1<\cdots<n_r}\frac{\left(\frac 12\right)_{n_1}}{n_1!}\frac{1}{\left(n_1+\frac 12\right)^{2k_1}\cdots\left(n_r+\frac 12\right)^{2k_r}}\frac{n_r!}{\left(\frac 12\right)_{n_r}}\frac{\left(n_1+\frac 12\right)+\left(n_r+\frac 12\right)}{n_r+\frac 12}\\
&=\sum_{2\leq r}(st)^{2r-4}\sum_{0\leq n_1<\cdots<n_r}\frac{(n_1+n_r+1)\left(\frac 12\right)_{n_1}}{n_1!}\frac{1}{\left(\left(n_1+\frac 12\right)^2-s^2\right)\cdots\left(\left(n_r+\frac 12\right)^2-s^2\right)}\frac{n_r!}{\left(\frac 12\right)_{n_r+1}}\\
&=\sum_{0\leq r}(st)^{2r}\sum_{0\leq n<m}\frac{(n+m+1)\left(\frac 12\right)_{n}}{n!\left(\left(n+\frac 12\right)^2-s^2\right)}\frac{m!}{\left(\left(m+\frac 12\right)^2-s^2\right)\left(\frac 12\right)_{m+1}}\\
&\qquad\cdot\sum_{n<n_1<\cdots<n_r<m}\frac{1}{\left(\left(n_1+\frac 12\right)^2-s^2\right)\cdots\left(\left(n_r+\frac 12\right)^2-s^2\right)}.
\end{align*}
Using the identity
\begin{align*}
&\sum_{0\leq r}(st)^{2r}\sum_{n<n_1<\cdots<n_r<m}\frac{1}{\left(\left(n_1+\frac 12\right)^2-s^2\right)\cdots\left(\left(n_r+\frac 12\right)^2-s^2\right)}\\
&=\prod_{n<k<m}\left(1+\frac{(st)^2}{\left(k+\frac 12\right)^2-s^2}\right)\\
&=\prod_{n<k<m}\frac{\left(k+\frac 12-s\sqrt{1-t^2}\right)\left(k+\frac 12+s\sqrt{1-t^2}\right)}{\left(k+\frac 12-s\right)\left(k+\frac 12+s\right)}\\
&=\frac{\left(\frac 12-s,\frac 12+s\right)_{n+1}}{\left(\frac 12-s\sqrt{1-t^2},\frac 12+s\sqrt{1-t^2}\right)_{n+1}}\frac{\left(\frac 12-s\sqrt{1-t^2},\frac 12+s\sqrt{1-t^2}\right)_m}{\left(\frac 12-s,\frac 12+s\right)_m},
\end{align*}
we obtain
\begin{align*}
&\sum_{2\leq r\leq k}s^{2k-4}t^{2r-4}\sum_{\substack{1\leq k_1,\dots,k_r\\k_1+\cdots+k_r=k}}(Z(2k_1,\dots,2k_r)+Z(2k_1-1,2k_2,\dots,2k_{r-1},2k_r+1))\\
&=\sum_{0\leq n<m}\frac{(n+m+1)\left(\frac 12\right)_{n}}{n!\left(\left(n+\frac 12\right)^2-s^2\right)}\frac{m!}{\left(\left(m+\frac 12\right)^2-s^2\right)\left(\frac 12\right)_{m+1}}\\
&\qquad\cdot\frac{\left(\frac 12-s,\frac 12+s\right)_{n+1}}{\left(\frac 12-s\sqrt{1-t^2},\frac 12+s\sqrt{1-t^2}\right)_{n+1}}\frac{\left(\frac 12-s\sqrt{1-t^2},\frac 12+s\sqrt{1-t^2}\right)_m}{\left(\frac 12-s,\frac 12+s\right)_m}\\
&=\sum_{0\leq n<m}\frac{(n+m+1)\left(\frac 12,\frac 12-s,\frac 12+s\right)_{n}}{n!\left(\frac 12-s\sqrt{1-t^2},\frac 12+s\sqrt{1-t^2}\right)_{n+1}}\frac{m!\left(\frac 12-s\sqrt{1-t^2},\frac 12+s\sqrt{1-t^2}\right)_m}{\left(\frac 12,\frac 12-s,\frac 12+s\right)_{m+1}}\\
&=\sum_{0\leq n\leq m}\frac{(n+m+1)\left(\frac 12,\frac 12-s,\frac 12+s\right)_{n}}{n!\left(\frac 12-s\sqrt{1-t^2},\frac 12+s\sqrt{1-t^2}\right)_{n+1}}\frac{m!\left(\frac 12-s\sqrt{1-t^2},\frac 12+s\sqrt{1-t^2}\right)_m}{\left(\frac 12,\frac 12-s,\frac 12+s\right)_{m+1}}\\
&\qquad-2\sum_{0\leq n}\frac 1{\left(\left(n+\frac 12\right)^2-s^2\right)\left(\left(n+\frac 12\right)^2-s^2(1-t^2)\right)}
\end{align*}
Applying Theorem \ref{thm:double_whipple} with $a=\frac 12-s,b=\frac 12+s,c=\frac 12,d=\frac 12-s\sqrt{1-t^2},e=\frac 12+s\sqrt{1-t^2}$, we obtain
\begin{align*}
&\sum_{2\leq r\leq k}s^{2k-4}t^{2r-4}\sum_{\substack{1\leq k_1,\dots,k_r\\k_1+\cdots+k_r=k}}(Z(2k_1,\dots,2k_r)+Z(2k_1-1,2k_2,\dots,2k_{r-1},2k_r+1))\\
&=\frac{\sin\left(\pi s\sqrt{1-t^2}\right)}{s\sqrt{1-t^2}}\sum_{0\leq n}\frac{(-1)^n(2n+1)}{\left(\left(n+\frac 12\right)^2-s^2\right)\left(\left(n+\frac 12\right)^2-s^2(1-t^2)\right)}\\
&\qquad -2\sum_{0\leq n}\frac 1{\left(\left(n+\frac 12\right)^2-s^2\right)\left(\left(n+\frac 12\right)^2-s^2(1-t^2)\right)}\\
&=\frac{\sin\left(\pi s\sqrt{1-t^2}\right)}{s^3t^2\sqrt{1-t^2}}\sum_{0\leq n}(-1)^n(2n+1)\left(\frac{1}{\left(n+\frac 12\right)^2-s^2}-\frac 1{\left(n+\frac 12\right)^2-s^2(1-t^2)}\right)\\
&\qquad-\frac{2}{s^2t^2}\sum_{0\leq n}\left(\frac 1{\left(n+\frac 12\right)^2-s^2}-\frac 1{\left(n+\frac 12\right)^2-s^2(1-t^2)}\right)\\
&=\frac{\sin\left(\pi s\sqrt{1-t^2}\right)}{s^3t^2\sqrt{1-t^2}}\left(\frac{\pi}{\cos\pi s}-\frac{\pi}{\cos\left(\pi s\sqrt{1-t^2}\right)}\right)-\frac{\pi\tan\pi s}{s^3t^2}+\frac{\pi\tan\left(\pi s\sqrt{1-t^2}\right)}{s^3t^2\sqrt{1-t^2}}\\
&=\frac{\pi\sin\left(\pi s\sqrt{1-t^2}\right)}{s^3t^2\sqrt{1-t^2}\cos\pi s}-\frac{\pi\tan\pi s}{s^3t^2}\\
\end{align*}
Here, we have used the well-known identities
\begin{align*}
\sum_{0\leq n}\frac{1}{\left(n+\frac 12\right)^2-X^2}&=\frac{\pi\tan\pi X}{2X}\\
\sum_{0\leq n}\frac{(-1)^n(2n+1)}{\left(n+\frac 12\right)^2-X^2}&=\frac{\pi}{\cos\pi X}.
\end{align*}
This completes the proof.
\end{proof}
We also use the following lemma due to Hoffman \cite{h} in the proof of Theorem \ref{thm:main}.
\begin{Lem}[Hoffman {\cite[Lemma 4]{h}}]\label{lem:hoffman}
The following identity holds as an identity of formal power series in $X$ and $Y$:
\[
\frac{\sin\left(\pi\sqrt{X}\sqrt{1-Y}\right)}{\pi\sqrt{X}\sqrt{1-Y}}=\sum_{0\leq k}Y^k\left(P_k(\pi^2 X)\cos\pi\sqrt{X}+Q_k(\pi^2X)\frac{\sin\pi\sqrt{X}}{\pi\sqrt{X}}\right).
\]
Here,
\begin{align*}
P_k(x)&\coloneqq-\frac 1{2^{2k-1}}\sum_{j=0}^{\lfloor\frac{k-1}2\rfloor}\frac{(-4x)^j}{(2j+1)!}\binom{2k-2j-1}k\\
Q_k(x)&\coloneqq\frac 1{2^{2k}}\sum_{j=0}^{\lfloor\frac{k}2\rfloor}\frac{(-4x)^j}{(2j)!}\binom{2k-2j}k.
\end{align*}
\end{Lem}

\begin{proof}[Proof of Theorem \ref{thm:main}]
By Lemmas \ref{lem:1} and \ref{lem:hoffman}, we have
\begin{align*}
&\sum_{2\leq r\leq k}s^{2k-4}t^{2r-4}\sum_{\substack{1\leq k_1,\dots,k_r\\k_1+\cdots+k_r=k}}(Z(2k_1,\dots,2k_r)+Z(2k_1-1,2k_2,\dots,2k_{r-1},2k_r+1))\\
&=\frac{\pi\sin\left(\pi s\sqrt{1-t^2}\right)}{s^3t^2\sqrt{1-t^2}\cos\pi s}-\frac{\pi\tan\pi s}{s^3t^2}\\
&=\frac{\pi^2}{s^2}\sum_{2\leq r}t^{2r-4}\left(P_{r-1}(\pi^2s^2)+Q_{r-1}(\pi^2 s^2)\frac{\tan\pi s}{\pi s}\right).
\end{align*}
Comparing coefficients of $t^{2r-4}$ for $r\geq 2$, we obtain
\begin{align*}
&\sum_{r\leq k}s^{2k-4}\sum_{\substack{1\leq k_1,\dots,k_r\\k_1+\cdots+k_r=k}}(Z(2k_1,\dots,2k_r)+Z(2k_1-1,2k_2,\dots,2k_{r-1},2k_r+1))\\
&=\frac{\pi^2}{s^2}\left(P_{r-1}(\pi^2s^2)+Q_{r-1}(\pi^2 s^2)\frac{\tan\pi s}{\pi s}\right)\\
&=\frac{\pi^2}{s^2}P_{r-1}(\pi^2s^2)+2Q_{r-1}(\pi^2 s^2)\sum_{1\leq k}Z(2k)s^{2k-4}.
\end{align*}
Next, comparing coefficients of $s^{2k-4}$ for $k\geq r$, we obtain
\begin{align*}
&\sum_{\substack{1\leq k_1,\dots,k_r\\k_1+\cdots+k_r=k}}(Z(2k_1,\dots,2k_r)+Z(2k_1-1,2k_2,\dots,2k_{r-1},2k_r+1))\\
&=\frac 1{2^{2r-3}}\sum_{j=0}^{\lfloor{\frac{r-1}2}\rfloor}\frac{(-4\pi^2)^j}{(2j)!}\binom{2r-2j-2}{r-1}Z(2k-2j).
\end{align*}
This completes the proof.
\end{proof}

\end{document}